\documentclass[reqno]{amsart}
\usepackage{amssymb}
\usepackage[mathscr]{eucal}
\usepackage{hyperref}
\usepackage{graphicx}
\usepackage[all]{xy}

\theoremstyle{plain}
\newtheorem{prop}{Proposition}[section]
\newtheorem{thm}[prop]{Theorem}
\newtheorem{cor}[prop]{Corollary}
\newtheorem{lem}[prop]{Lemma}

\theoremstyle{definition}

\newtheorem{rem}[prop]{Remark}
\newtheorem{rems}[prop]{Remarks}
\newtheorem{example}[prop]{Example}
\newtheorem{lab}[prop]{}

\newcommand{\isoto}{\overset{\sim}{\to}}

\newcommand{\A}{{\mathbb{A}}}
\newcommand{\C}{{\mathbb{C}}}
\newcommand{\N}{{\mathbb{N}}}

\newcommand{\R}{{\mathbb{R}}}

\newcommand{\m}{{\mathfrak{m}}}

\newcommand{\scrO}{{\mathscr{O}}}
\newcommand{\scrP}{{\mathscr{P}}}

\newcommand{\sfS}{\mathsf{S}}

\DeclareMathOperator{\ch}{char}
\DeclareMathOperator{\Hom}{Hom}

\DeclareMathOperator{\Spec}{Spec}

\DeclareMathOperator{\tr}{tr}

\newcommand{\cone}{\mathrm{cone}}
\newcommand{\conv}{\mathrm{conv}}
\newcommand{\reg}{\mathrm{reg}}
\newcommand{\Sym}{\mathrm{Sym}}

\newcommand{\comp}{\mathbin{\scriptstyle\circ}}
\newcommand{\du}{{\scriptscriptstyle\vee}}
\renewcommand{\setminus}{\smallsetminus}
\newcommand{\ol}{\overline}
\newcommand{\plus}{{\scriptscriptstyle+}}
\newcommand{\wh}[1]{\widehat{#1}}

\newcommand{\all}{\forall\,}
\newcommand{\ex}{\exists\,}
\renewcommand{\subset}{\subseteq}
\renewcommand{\supset}{\supseteq}

\renewcommand{\choose}[2]{\genfrac(){0pt}{}{#1}{#2}}
\newcommand{\bil}[2]{\langle{#1},{#2}\rangle}

\newcommand{\sa}{semi-alge\-braic}

\begin{document}

\title
[Spectrahedral shadows]
{Spectrahedral shadows}

\begin{abstract}
We show that there are many (compact) convex semi-alge\-braic sets in
euclidean space that are not spectrahedral shadows. This
gives a negative answer to a question by Nemirovski, resp.\ it shows
that the Helton-Nie conjecture is false.
\end{abstract}

\author{Claus Scheiderer}
\address
  {Fachbereich Mathematik und Statistik \\
  Universit\"at Konstanz \\
  78457 Konstanz \\
  Germany}
\email
  {claus.scheiderer@uni-konstanz.de}

\keywords
  {Spectrahedral shadows, semidefinite representations, semidefinite
  programming, Helton-Nie conjecture, moment relaxation, convex
  algebraic geometry, real algebraic geometry}

\subjclass[2010]
  {Primary
  90C22,
  secondary
  14P05}

\maketitle


\section*{Introduction}

Semidefinite programming is a far-reaching generalization of linear
programming. While a linear program optimizes a linear function over
a polyhedron, a semi\-definite program optimizes it over a convex
region described by symmetric linear matrix inequalities. Under mild
conditions,
semidefinite programs can be solved in polynomial time up to any
prescribed accuracy.
They have numerous applications in applied mathematics, engineering,
control theory and so forth (see \cite[Chapter~1]{al}).

The feasible regions of semidefinite programs are called
spectrahedral shadows, or also semidefinitely (or SDP) representable
sets.
These are the sets $K\subset\R^n$ that can be written
\begin{equation}\label{semidefrep}%
K\>=\>\Bigl\{\xi\in\R^n\colon\ex\eta\in\R^m\ A+\sum_{i=1}^n\xi_iB_i
+\sum_{j=1}^m\eta_jC_j\succeq0\Bigr\}
\end{equation}
where $m\ge0$, $A,\,B_i,\,C_j$ are real symmetric matrices of the
same size and $M\succeq0$ means that $M$ is positive semidefinite.
Any representation as in \eqref{semidefrep} is called a semidefinite
representation of~$K$.

There has been considerable interest in characterizing
spectrahedral shadows by their geometric properties.
Essentially,
this is the question of what problems in optimization can be modeled
as semidefinite programs. Nemirovski
\cite{ne} in his 2006 plenary address at ICM~Madrid remarked:
``A~seemingly interesting question is to characterize
SDP-representable sets. Clearly, such a set is convex and
semi-algebraic. Is the inverse also true? (\dots) This question
seems to be completely open.'' Helton and Nie (\cite[p.~790]{hn1})
conjectured that the answer is in fact yes, i.e., that every convex
\sa\ set in $\R^n$ is a spectrahedral shadow.

Although the general question has so far been elusive, many results
have been obtained in support of the Helton-Nie conjecture. The class
of spectrahedral shadows is known to be closed under
taking linear images or preimages, finite intersections, or convex
hulls of finite unions (\cite{hn1}, \cite{nsi}). It is also closed
under convex duality respectively polarity, and under taking
topological closures. Helton and
Nie (\cite{hn1}, \cite{hn2}) gave a series of sufficient conditions
for semidefinite representability of a convex \sa\ set $K$, in terms
of curvature conditions for the boundary. Roughly, their results are
saying that when $K$ is compact and its boundary is sufficiently
nonsingular and has strictly positive curvature, then $K$ is a
spectrahedral shadow.
Netzer \cite{nt} proved that the interior of a spectrahedral shadow
is again a spectrahedral shadow, and more generally, that removing
suitably parametrized families of faces from a spectrahedral shadow
results again in a spectrahedral shadow.
By applying the criteria of Helton-Nie,
Netzer and Sanyal \cite{ns} showed that smooth hyperbolicity cones
are spectrahedral shadows.
Scheiderer \cite{sch:sdpcurv} showed that closed convex hulls of
one-dimensional \sa\ sets are always spectrahedral shadows,
and that the Helton-Nie conjecture is true for subsets of the plane.

In addition there are plenty of further results on semidefinite
representations for particular kinds of sets. See, for example,
\cite{gn},
\cite{gpt},
\cite{gwz},
\cite{he},
\cite{nps},
\cite{ni},
\cite{nipast},
\cite{sp},
\cite{spw},
\cite{sch:conv},
and see
\cite{nene},
\cite{btn},
\cite{ne},
\cite[ch.~6]{bpt}
or \cite[ch.~2, 4 and~5]{al}
for surveys on semidefinite representation.

An important general technique for constructing semidefinite
representations was introduced by Lasserre \cite{la}, and
independently by Parrilo \cite{pa}. It is based on a dual relaxation
principle
and is generally known as the moment relaxation method. Starting
with a (basic closed) \sa\ set $S\subset\R^n$, it produces outer
approximations of the convex hull of $S$ that have explicit
semidefinite representations. When $S$ is compact, these
approximations can be made arbitrarily close. Under favorable
conditions, moment relaxation is known to become exact, meaning that
a suitable such approximation coincides with the convex hull of $S$,
up to taking closures.

In this paper we exhibit, for the first time, non-trivial conditions
that are necessary for semidefinite representability. They are based
on semidefinite duality, and they imply that there are no more closed
spectrahedral shadows
than those obtainable from exact moment relaxation in a generalized
sense.
We then use arguments from algebraic geometry, in particular
properties of smooth morphisms of varieties, to show that these
conditions are indeed non-trivial, and to produce
concrete examples of convex sets that fail to be spectrahedral
shadows.
Among them are natural prominent sets like the cone
of non-negative forms of fixed degree in $\R[x_1,\dots,x_n]$, in
every case where this cone is different from the sums of squares cone
(Corollary \ref{psdcone}).
In fact, for every \sa\ set $S\subset\R^n$ of dimension at least two
we prove that there exist polynomial maps $\varphi\colon\R^n\to\R^m$
for which the closed convex hull of $\varphi(S)$ in $\R^m$ has no
semidefinite representation. This is in marked contrast to the case
where $S$ has dimension one, when it is known that the closed convex
hull of $S$ is always a spectrahedral shadow
\cite{sch:sdpcurv}.

For optimization, our results imply that there exist natural
\sa\ optimization problems that cannot be modeled exactly as
semidefinite programs. For example, the problem of minimizing a
general polynomial of degree $d\ge4$ in $n\ge3$ variables (or of
degree $d\ge6$ in $n=2$ variables) over the unit ball in $\R^n$ is of
this sort.

The paper is organized as follows. In Section~2 we recall and
generalize the moment relaxation construction, arriving at general
sufficient conditions for semidefinite representability. In Section~3
we show that the conditions obtained in Section~2 are also necessary.
The main result is Theorem \ref{maininhom}. In Section~4 we present
concrete constructions of closed convex sets that violate the
necessary conditions from Section~3, and we give a few explicit
examples. Finally, Section~5 contains a number of open questions.
\smallskip

\textbf{Acknowledgement.}
The elegant construction for the proof of Proposition \ref{kerneu}
was suggested by Christoph Hanselka. I am most grateful to him for
his kind permission to include his argument here. I am also indebted
to Tim Netzer for comments on a first preliminary version.


\section{Preliminaries and notation}

\begin{lab}
A symmetric matrix $A\in\Sym_d(\R)$ is said to be positive
semidefinite (psd), denoted $A\succeq0$, if all its eigenvalues are
non-negative. If in addition all eigenvalues are nonzero then $A$ is
positive definite, written $A\succ0$. The canonical inner product on
$\Sym_d(\R)$ is denoted $\bil AB=\tr(AB)$, for $A,\,B\in\Sym_d(\R)$.
The set $\Sym_d^\plus(\R)=\{A\in\Sym_d(\R)\colon A\succeq0\}$ is a
closed convex cone in $\Sym_d(\R)$, self-dual with respect to the
inner product. The same terminology applies when the field $\R$ of
real numbers is replaced by a real closed field~$R$.
\end{lab}

\begin{lab}
A set $K\subset\R^n$ is called a \emph{spectrahedron} if there exist
$d\ge1$ and $M_0,\dots,M_n\in\Sym_d(\R)$ such that $K=\{\xi\in\R^n
\colon M_0+\sum_{i=1}^n\xi_iM_i\succeq0\}$. The set $K$ is said to be
a \emph{spectrahedral shadow} (or to have a semidefinite
representation)
if there exists a spectrahedron $S\subset\R^m$ for some $m$ and a
linear map $f\colon\R^m\to\R^n$ such that $K=f(S)$.

Spectrahedra have also been called LMI-sets or LMI-representable
sets. In the plane, spectrahedra are characterized by the former
Lax conjecture, which has been proved by Helton-Vinnikov \cite{hv} in
2007. In higher dimension there exist only conjectural
characterizations of spectrahedra (so-called generalized Lax
conjecture).

Spectrahedral shadows
have as well occured under various different names, such as projected
spectrahedra, SDP representable sets or lifted-LMI representable
sets.
These sets are convex and \sa, but so far no other restrictions were
known.
\end{lab}

\begin{lab}
For $V$ a vector space over a field $k$ we denote the dual space of
$V$ by $V^\du=\Hom_k(V,k)$. Let $V$ be a finite-dimensional
$\R$-vector space. By a (convex) cone $C$ in $V$ we mean a non-empty
set $C\subset V$ with $C+C\subset C$ and $aC\subset C$ for all real
numbers $a\ge0$. Given any set $M\subset V$ let $\conv(M)$ denote the
convex hull of $M$, and let $\cone(M)$ be the convex cone generated
by $M$ (consisting of all finite linear combinations of elements of
$M$ with non-negative coefficients). Moreover, $M^*\subset V^\du=
\Hom_\R(V,\R)$ denotes the (closed convex) cone dual to $M$, i.e.\
$M^*=\{\lambda\in V^\du\colon\all x\in M$ $\lambda(x)\ge0\}$. When
$M$ is a \sa\ set then so are $\conv(M)$ and $\cone(M)$, by
Carath\'eodory's lemma, and also $M^*$.
\end{lab}

\begin{lab}\label{soritvarieties2}%
Given a field $k$, a $k$-algebra is a commutative ring $A$ with a
fixed ring homomorphism $k\to A$. The $k$-algebra $A$ is said to be
finitely generated if it is finitely generated as a ring over~$k$.
The ideal in a ring $A$ generated by a family of elements $a_i\in A$
($i\in I$) is denoted $\langle a_i\colon i\in I\rangle$.

We use standard terminology for algebraic varieties, that we briefly
recall. Generally, we use the language of schemes. See
\ref{ag4dummies} below for informal rephrasings of the most important
concepts in non-technical language. Note however that it would be
most cumbersome and awkward to formulate parts of
Section~\ref{counterex} in such language, which is why we will not
make such an attempt.

Our fields $k$ will be real closed fields, and in particular they
have characteristic zero. All our $k$-varieties will be affine, and
we assume them to be reduced but not necessarily irreducible. Thus,
an affine $k$-variety $V$ is the spectrum of a reduced finitely
generated $k$-algebra $A$, and $A=k[V]=\Gamma(V,\scrO_V)$ is the
affine coordinate ring of~$V$. As usual, $V(k)=\Hom_k(A,k)$ is the
set of $k$-rational points of $V$. Given $\xi\in V(k)$, the local
ring of $V$ at $\xi$ is denoted $\scrO_{V,\xi}$. By $\m_{V,\xi}$ we
invariantly denote both the maximal ideal of $\scrO_{V,\xi}$ and its
preimage in $k[V]$. Given a morphism $\phi\colon V\to W$ of affine
$k$-varieties, the associated homomorphism $k[W]\to k[V]$ of
$k$-algebras is denoted~$\phi^*$. Given a field extension $K/k$ we
write $V_K:=V\times_k\Spec(K)$ for the base field extension of $V$,
and similarly $\phi_K\colon V_K\to W_K$ for the base field extension
of~$\phi$. Upon identifying $K[V]=k[V]\otimes_kK$ we have
$\phi_K^*=\phi^*\otimes1$ for the induced map $K[Y]\to K[X]$.
\end{lab}

\begin{lab}\label{ag4dummies}%
To make the paper more accessible to readers who are not familiar
with the basic notions of algebraic geometry, here are some
explanations. Let $k$ be either be the field $\R$ of real numbers or
a real closed extension thereof, and let $\ol k$ be an algebraic
closure of~$k$. Informally speaking, an affine $k$-variety can be
thought of as a subset $V\subset\ol k^n$ (for some $n$) that can be
described by polynomial equations with coefficients in~$k$, and that
is equipped with the $k$-Zariski topology (the closed subsets of
which are the affine $k$-varieties contained in $V$). Given a second
affine $k$-variety $W\subset\ol k^m$, a morphism $\phi\colon V\to W$
of $k$-varieties corresponds to a map $V\to W$ given component-wise
by $k$-polynomials. Associated with $V$ is its affine coordinate ring
$k[V]=k[x_1,\dots,x_n]/I_V$, where $I_V$ is the ideal of polynomials
vanishing identically on $V$. Associated with $\phi$ is the
(``pull-back'') ring homomorphism $\phi^*\colon k[W]\to k[V]$ induced
by $\phi$ in the obvious way. The set of $k$-rational points of $V$
is $V(k):=k^n\cap V$. Given $\xi\in V(k)$, let $\m_{V,\xi}=
\{f\in k[V]\colon f(\xi)=0\}$, a maximal ideal of $k[V]$. The local
ring of $V$ at $\xi$ is by definition the localization
$\scrO_{V,\xi}=k[V]_{\m_{V,\xi}}$ of $k[V]$ at its maximal ideal
$\m_{V,\xi}$. Given a field extension $K/k$, the base field extension
$V_K$ of $V$ is the affine $K$-variety in $\ol K^n$ defined by the
same $k$-polynomials as $V$, and it has coordinate ring
$K[V]=k[V]\otimes_kK$.

For example, affine $n$-space is the affine $k$-variety
$\A^n=\ol k^n$ and has $k[\A^n]=k[x_1,\dots,x_n]=k[x]$ and
$\A^n(k)=k^n$. Given $\xi\in k^n$, the local ring $\scrO_{\A^n,\xi}$
is the ring of fractions $\frac fg$ with $f,\,g\in k[x]$ and
$g(\xi)\ne0$.
\end{lab}

\begin{lab}
When $V$ is an affine variety over $k=\R$, we will always equip the
set $V(\R)$ of $\R$-rational points with the euclidean topology
(induced by $V(\R)\subset\R^n$ if $V\subset\A^n$ is a closed
subvariety). This topology is independent of the choice of a closed
embedding $V\subset\A^n$. A subset $S\subset V(\R)$ is called \sa\
if it can be
written as a finite boolean combination of sets of the form
$\{\xi\in V(\R)\colon f(\xi)>0\}$ with $f\in\R[V]$.
\end{lab}

\begin{lab}\label{linsystmorph}%
Let $V$ be an affine $k$-variety, and let $L\subset k[V]$ be a
$k$-linear subspace of finite dimension. Let $\sfS^\bullet L=
\bigoplus_{d\ge0}\sfS^dL$ be the symmetric $k$-algebra over $L$, and
let $\sfS^\bullet L\to k[V]$ be the natural $k$-homomorphism induced
by the inclusion $L\subset k[V]$. The associated morphism of affine
$k$-varieties will be denoted $\varphi_L\colon V\to\A_L$, where
$\A_L:=\Spec(\sfS^\bullet L)$ is the affine space with coordinate
ring $\sfS^\bullet L$. In plainer terms, upon fixing a linear basis
$g_1,\dots,g_m$ of $L$, we may identify $\varphi_L$ with the map
$V\to\A^m$ given by $\xi\mapsto(g_1(\xi),\dots,g_m(\xi))$. Note that
$\A_L(k)=L^\du=\Hom_k(L,k)$, the linear space dual to~$L$, so that on
$k$-rational points the map $\varphi_L\colon V(k)\to L^\du$ sends
$\xi\in V(k)$ to the evaluation map $L\to k$ at~$\xi$.
\end{lab}


\section{Sufficient conditions for semidefinite representability}

Let $V$ be an affine $\R$-variety. Given a \sa\ subset $S\subset
V(\R)$ we write $\scrP(S)=\{f\in\R[V]\colon f\ge0$ on~$S\}$.

\begin{lab}\label{recalllass}%
We start by informally recalling the moment relaxation construction,
due to Lasserre \cite{la} and independently Parrilo \cite{pa}.
Let $L\subset\R[V]$ be a linear subspace with $\dim(L)=m<\infty$, and
let $\varphi_L\colon V\to\A_L\cong\A^m$ be the associated morphism,
see \ref{linsystmorph}. Assume that $S\subset V(\R)$ is a basic
closed \sa\ set, say $S=\{\xi\in V(\R)\colon h_i(\xi)\ge0$
($i=1,\dots,r)\}$ where $h_1,\dots,h_r\in\R[V]$. We are trying to
find a semidefinite representation of the convex hull $K$ of
$\varphi_L(S)$ in $\A_L(\R)=L^\du\cong\R^m$, or at least an
approximate such representation.

Without any serious restriction we can assume $1\notin L$.
Fix a sequence $W_0,\dots,W_r$ of finite-dimensional linear
subspaces of $\R[V]$. Any $f\in L_1:=\R1+L\subset\R[V]$ that has a
representation $f=s_0+\sum_{i=1}^rs_ih_i$ with $s_i$ a sum of squares
of elements of $W_i$ ($i=0,\dots,r$) is obviously non-negative on $S$.
So the set of all such $f$ is a convex cone $C=C(W_0,\dots,W_r)$,
contained in $L_1\cap\scrP(S)$. By construction, the dual cone
$C^*\subset L_1^\du$ has an explicit semidefinite representation.
For $\lambda\in L^\du$ let $\lambda'\in L_1^\du$ be defined by
$\lambda'|_L=\lambda$ and by $\lambda'(1)=1$. The set
$K'=K'(W_0,\dots,W_r)$ of all $\lambda\in\A_L(\R)=L^\du$ for which
$\lambda'\in C^*$ is a closed spectrahedral shadow that contains $K$.
Enlarging the spaces $W_i$, or adding more inequalities $h_i$ to the
description of $S$, results in $K'$ getting smaller,
and therefore becoming a closer approximation to~$K$. Of particular
interest is the case where $C=L_1\cap\scrP(S)$.
This condition is usually rephrased by saying that the linear
polynomials non-negative on $\varphi_L(S)$ (i.e.\ the elements of
$L_1\cap\scrP(S)$) have weighted sum of squares representations of
uniformly bounded degrees (the ``degree bounds'' being given by the
subspaces $W_i$). The moment relaxation is \emph{exact} in this case,
which means that $K'=\ol K$, the closure of $K$. Therefore, under
the assumption $C=L_1\cap\scrP(S)$, the closure $\ol K$ is a
spectrahedral shadow.
\end{lab}

We now generalize this procedure, to arrive at a general sufficient
condition for semidefinite representability.
First two auxiliary lemmas.

\begin{lem}\label{closedconvhullsorit}%
Let $V$ be an affine $\R$-variety, and let $S\subset V(\R)$ be a \sa\
set. Let $L\subset\R[V]$ be a finite-dimensional linear subspace with
$1\notin L$, and write $L_1:=\R1+L$.
\begin{itemize}
\item[(a)]
$\varphi_L(S)$ is a \sa\ subset of $L^\du$.
\item[(b)]
The closed convex hull $\ol{\conv(\varphi_L(S))}$ of $\varphi_L(S)$
in $L^\du$ consists of all $\lambda\in L^\du$ that satisfy
$\lambda'(g)\ge0$ for every $g\in L_1\cap\scrP(S)$.
\item[(c)]
The closed conic hull $\ol{\cone(\varphi_L(S))}$ of $\varphi_L(S)$
in $L^\du$ consists of all $\lambda\in L^\du$ that satisfy
$\lambda(g)\ge0$ for every $g\in L\cap\scrP(S)$.
\end{itemize}
\end{lem}

\begin{proof}
In (b), $\lambda'\in L_1^\du$ denotes the extension of $\lambda\in
L^\du$ defined by $\lambda'(1)=1$, see \ref{recalllass}. (a) follows
from the Tarski-Seidenberg theorem, and (b), (c) are consequences of
convex duality.
\end{proof}

\begin{lem}\label{sosistsdp}%
Let $A$ be an $\R$-algebra, let $U\subset A$ be a linear subspace
with $\dim(U)<\infty$, and let $UU$ be the linear subspace of $A$
spanned by all products $uu'$ ($u,\,u'\in U$). Then the cone
$\Sigma U^2$,
consisting of all finite sums of squares of elements of $U$, is a
spectrahedral shadow
in~$UU$.
\end{lem}

\begin{proof}
Choose a linear basis $u_1,\dots,u_n$ of $U$. The linear map
$f\colon\Sym_n(\R)\to UU$, $(a_{ij})\mapsto\sum_{i,j}a_{ij}u_iu_j$
satisfies $\Sigma U^2=f(\Sym_n^\plus(\R))$, which shows the claim.
\end{proof}

We keep fixing an affine $\R$-variety $V$, a \sa\ set $S\subset
V(\R)$ and a finite-dimensional linear subspace $L\subset\R[V]$. As
before write $L_1=L+\R1$.

\begin{prop}\label{umker}%
Let $\phi_i\colon X_i\to V$ ($i=1,\dots,m$) be finitely many
morphisms of affine $\R$-varieties. For every $i=1,\dots,m$ let
$U_i\subset\R[X_i]$ be a finite-dimensional linear subspace, and
assume that the following two conditions hold:
\begin{itemize}
\item[(1)]
$S\subset\phi_i(X_i(\R))$ for $i=1,\dots,m$;
\item[(2)]
for every $f\in L_1\cap\scrP(S)$ there exists $i\in\{1,\dots,m\}$
such that $\phi_i^*(f)\in\R[X_i]$ is a sum of squares of elements
of $U_i$ (in $\R[X_i]$).
\end{itemize}
Then $\ol{\conv(\varphi_L(S))}$, the closed convex hull of
$\varphi_L(S)$ in $\A_L(\R)=L^\du$, is a spectrahedral shadow.
\end{prop}

\begin{proof}
Write $C:=L_1\cap\scrP(S)$, which is a closed convex cone in $L_1$.
For a given index $i\in\{1,\dots,m\}$ let $C_i\subset L_1$ be the
cone of all $f\in L_1$ for which $\phi_i^*(f)$ is a sum of squares of
elements of $U_i$ in $\R[X_i]$. By Lemma \ref{sosistsdp}, and since
linear preimages of spectrahedral shadows are again spectrahedral
shadows,
$C_i$ is a spectrahedral shadow in~$L_1$.
By condition (1), elements of $C_i$ are non-negative on $S$, which
means $C_i\subset C$.
Therefore $C=\bigcup_{i=1}^mC_i$ by~(2), and hence we have
$C^*=\bigcap_{i=1}^mC_i^*$ for the dual cones.
For every index $i$ the cone $C_i^*$, being the dual cone to a
spectrahedral shadow cone, is itself a spectrahedral shadow. So it
follows that $C^*$ is a spectrahedral shadow in $L_1^\du$.

For the convex hull $K:=\conv(\varphi_L(S))\subset L^\du$ we have
$\ol K=\{\lambda\in L^\du\colon\lambda'\in C^*\}$, see Lemma
\ref{closedconvhullsorit}(b).
So $\ol K$ is the preimage of the spectrahedral shadow $C^*$ under
the affine-linear map $L^\du\to L_1^\du$, $\lambda\mapsto\lambda'$,
and hence is a spectrahedral cone, as asserted.
\end{proof}

\begin{cor}\label{homumker}%
If, in Proposition \ref{umker}, condition (2) is only required to
hold for every $f\in L\cap\scrP(S)$, then $\ol{\cone(\varphi_L(S))}$,
the closed convex cone generated by $\varphi_L(S)$, is a
spectrahedral shadow.
\end{cor}

\begin{proof}
The proof is completely analogous to the proof of \ref{umker},
defining the respective cones $C$ and $C_i$ to be subcones of $L$
instead of $L_1$, and applying Lemma \ref{closedconvhullsorit}(c).
\end{proof}

The following examples and remarks illustrate Proposition
\ref{umker}.

\begin{rems}
\hfil
\smallskip

1.\
Proposition \ref{umker} can be seen as a generalization of the moment
relaxation construction. To explain this, assume that we are in the
situation of \ref{recalllass}, in particular $S=\{\xi\in V(\R)\colon
h_i(\xi)\ge0$ $(i=1,\dots,r)\}$ with $h_i\in\R[V]$. Let $X$ be the
affine $\R$-variety
obtained by formally adjoining square roots of $h_1,\dots,h_r$ to
$\R[V]$, i.e.\ $\R[X]=\R[V][t_1,\dots,t_r]/\langle t_i^2-h_i$,
$i=1,\dots,r\rangle$, and let $\phi\colon X\to V$ be the natural map.
Then clearly $\phi(X(\R))=S$. If subspaces $W_i\subset\R[V]$ as in
\ref{recalllass} have been found such that the sufficient exactness
condition from \ref{recalllass} is satisfied, i.e.\ if $L_1\cap
\scrP(S)=C(W_0,\dots,W_r)$ (in the notation of \ref{recalllass}),
this implies that for every $f\in L_1\cap\scrP(S)$ the pull-back
$\phi^*(f)\in\R[X]$ is a sum of squares in $\R[X]$ of elements from
the subspace $U:=\phi^*(W_0)+\sum_{i=1}^r\phi^*(W_i)\sqrt{h_i}$ of
$\R[X]$. So under this assumption, the conditions of Proposition
\ref{umker} are fulfilled with $m=1$ and these particular choices of
$\phi$ and~$U$.%
\smallskip

2.\
Conversely, the more general construction of a semidefinite
representation in Proposition \ref{umker}
is achieved essentially by reduction to a construction of moment
relaxation type, as in \ref{recalllass}. We leave it to the reader to
make this statement precise.%
\smallskip

3.\
The proof of Proposition \ref{umker} is constructive in the following
sense. If the morphisms $\phi_i\colon X_i\to V$ as well as the linear
subspaces $U_i\subset\R[X_i]$ are given explicitly, we can deduce
from this data an explicit semidefinite representation of
$\ol{\conv(\varphi_L(S))}$.
\end{rems}

\begin{example}
Let $C$ be a nonsingular affine curve over $\R$ for which $C(\R)$ is
compact. Let $L\subset\R[C]$ be a finite-dimensional linear subspace,
and consider the associated map $\varphi_L\colon C(\R)\to\A_L(\R)=
L^\du$. By \cite[Corollary 4.4]{sch:sdpcurv} there exists a
finite-dimensional linear subspace
$U\subset\R[C]$ such that every $f\in L+\R1$ that is non-negative on
$C(\R)$ is a sum of squares of elements from $U$. Using this fact,
Proposition \ref{umker} applies with $m=1$ and $\phi\colon X\to C$
the identity map of $C$, showing that the convex hull of
$\varphi_L(C(\R))$ in $L^\du$ is a spectrahedral shadow.
(This consequence was already drawn in \cite{sch:sdpcurv}.)
\end{example}

\begin{rem}
Later (Remark \ref{uniform} below) we'll see that it is not enough in
Proposition \ref{umker} to replace condition (2) by the weaker
condition that every $f\in L_1\cap\scrP(S)$ becomes a sum of squares
in one of the $\R[X_i]$. Rather, it is essential that such sum of
squares representations exist in a uniform way.
\end{rem}


\section{Necessary conditions for semidefinite representability}

In the previous section we stated sufficient conditions for
semidefinite representability. We now show that these conditions are
also necessary.
In the sequel let $x=(x_1,\dots,x_n)$ be a tuple of variables. We
start by recalling one form of duality in semidefinite programming
(see \cite{ra}):

\begin{prop}\label{dualspectrcone}%
Let $M_1,\dots,M_n\in\Sym_d(\R)$, write $M(\xi)=\sum_{i=1}^n
\xi_iM_i$ for $\xi\in\R^n$, and let $C=\{\xi\in\R^n\colon M(\xi)
\succeq0\}$ be the associated spectrahedral cone. Assume that
$M(\xi^0)\succ0$ for some $\xi^0\in\R^n$. Then the dual cone of $C$
has the following semidefinite representation:
$$C^*\>=\>\Bigl\{\Bigl(\bil B{M_1},\dots,\bil B{M_n}\Bigr)\colon
B\in\Sym_d(\R),\ B\succeq0\Bigr\}\>\subset\>\R^n.\eqno\square$$
\end{prop}

\begin{prop}\label{kerneu}%
Assume that $S\subset\R^n$ is a \sa\ set for which the closed conical
hull $\ol{\cone(S)}\subset\R^n$ of $S$ is a spectrahedral shadow.
Then there exists a morphism $\phi\colon X\to\A^n$ of
affine $\R$-varieties, together with a finite-dimensional $\R$-linear
subspace $U$ of $\R[X]$, such that $S\subset\phi(X(\R))$ and the
following holds: For every homogeneous linear polynomial $f\in\R[x]$
with $f\ge0$ on $S$, the pull-back $\phi^*(f)\in\R[X]$ is a sum of
squares of elements from~$U$.
\end{prop}

Our original proof for Proposition \ref{kerneu} (see version~1 of
arxiv:1612.07048) was non-constructive and used a compactness
argument for the real spectrum. The following explicit construction
is much more elegant and transparent. It was suggested by Christoph
Hanselka, who kindly agreed that his argument may be included here.
Independently, the original approach may still have its merits, as
we plan to demonstrate in follow-up work.

\begin{proof}
Let $C=\cone(S)$, the convex cone generated by $S$ in $\R^n$.
We may assume that $\R^n$ is affinely spanned by $S$.
By assumption, $\ol C$ is the linear image of a
spectrahedron $T\subset\R^N$ under a linear map $\pi\colon\R^N\to
\R^n$, for some~$N$. We may assume that $T$ is a cone,
and we may replace $\R^N$ by the linear hull of $T$.
Then $T$ can be represented by a homogeneous linear matrix inequality
that is strictly feasible.
So we can assume that there are integers $d\ge1$ and $m\ge0$,
together with linear matrix pencils $M(x)=\sum_{i=1}^nx_iM_i$,
$N(y)=\sum_{j=1}^my_jN_j$ in $\Sym_d(\R)$,
such that
$$T\>=\>\bigl\{(\xi,\eta)\in\R^n\times\R^m\colon M(\xi)+N(\eta)
\succeq0\},$$
such that $\ol C=\pi(T)$ where $\pi(\xi,\eta)=\xi$, and such that
there exists $(\xi,\eta)\in\R^n\times\R^m$ with $M(\xi)+N(\eta)
\succ0$.

Consider the closed subvariety $X$ of $\A^n\times\A^m\times\Sym_d$
defined over $\R$ whose $\C$-points are the triples $(\xi,\eta,A)$,
where $A$ is a symmetric $d\times d$-matrix satisfying
$$A^2\>=\>\sum_{i=1}^n\xi_iM_i+\sum_{j=1}^m\eta_jN_j.$$
We shall denote the coordinate functions on $X$ by
$$\bigl(x_1,\dots,x_n;\>y_1,\dots,y_m;\>
(z_{\mu\nu})_{1\le\mu,\nu\le d}\bigr)\>=\>(x,y,Z)$$
with $z_{\mu\nu}=z_{\nu\mu}$ for $1\le\mu,\,\nu\le d$.
Let $\phi\colon X\to\A^m$ be the projection $\phi(\xi,\eta,A)=\xi$.
Then $\phi(X(\R))=\pi(T)=\ol C$, since a real symmetric matrix is
psd if and only if it is the square of some real symmetric matrix.
Let $U\subset\R[X]$ be the linear subspace spanned
by the coefficient functions $z_{\mu\nu}=z_{\nu\mu}$ ($1\le\mu,\nu\le
d$) of $Z$. We claim that the assertion of \ref{kerneu} holds with
these choices of $\phi$ and~$U$.

To see this, let $f=\sum_{i=1}^na_ix_i$ be a linear homogeneous
polynomial in $\R[x]=\R[x_1,\dots,x_n]$ with $f\ge0$ on $S$ (and
hence $f\ge0$ on~$\ol C$).
So the tuple $(a,0)=(a_1,\dots,a_n;\,0,\dots,0)\in\R^n\times\R^m$
lies in the dual cone $T^*$ of~$T$.
Since the linear matrix inequality is strictly feasible, there exists
$B\in\Sym_d(\R)$ with $B\succeq0$ such that $a_i=\bil B{M_i}$
($1\le i\le n$) and $0=\bil B{N_j}$ ($1\le j\le m$), by Proposition
\ref{dualspectrcone}. Let $V\in\Sym_d(\R)$ with $B=V^2$. Then, as an
element of $\R[X]$, $\phi^*(f)$ is equal to
$$\sum_{i=1}^n\bil B{M_i}x_i+\sum_{j=1}^n\bil B{N_j}y_j\>=\>
\bigl\langle B,\>M(x)+N(y)\bigr\rangle\>=\>\bil{V^2}{Z^2}\>=\>
\bil{ZV}{ZV}.$$
This means that
$$\phi^*(f)\>=\>\sum_{\mu,\,\nu=1}^d\bigl((ZV)_{\mu\nu}\bigr)^2$$
is a sum of squares in $\R[X]$ from the linear subspace $U\subset
\R[X]$.
\end{proof}

Combining Propositions \ref{homumker} and \ref{kerneu}, we therefore
get:

\begin{thm}\label{mainhom}%
Let $S\subset\R^n$ be a \sa\ set, and let $C=\cone(S)$ be the convex
cone in $\R^n$ generated by~$S$. The closure $\ol C$ is a
spectrahedral shadow
if and only if there exists a morphism
$\phi\colon X\to\A^n$ of affine $\R$-varieties, together with an
$\R$-linear subspace $U\subset\R[X]$ of finite dimension, such that
\begin{itemize}
\item[(1)]
$S\subset\phi(X(\R))$,
\item[(2)]
for every homogeneous linear polynomial $f\in\R[x_1,\dots,x_n]$ with
$f\ge0$ on $S$, the pull-back $\phi^*(f)\in\R[X]$ is a sum of squares
of elements from~$U$.
\end{itemize}
\end{thm}

\begin{proof}
The second condition is necessary for $\ol C$ to be a spectrahedral
shadow
by Proposition \ref{kerneu}, and it is sufficient by \ref{homumker}.
\end{proof}

Instead of working with convex cones we may also dehomogenize and
derive a non-homogeneous version from Theorem \ref{mainhom}.
Alternatively, we could as well have worked in an inhomogeneous
setting from the beginning:

\begin{thm}\label{maininhom}%
Let $S\subset\R^n$ be a \sa\ set, and let $K=\conv(S)$ be its convex
hull in $\R^n$. The closure $\ol K$ is a spectrahedral shadow
if and only if there exists a morphism $\phi\colon X\to\A^n$ of
affine $\R$-varieties and an $\R$-linear subspace $U\subset\R[X]$ of
finite dimension such that
\begin{itemize}
\item[(1)]
$S\subset\phi(X(\R))$,
\item[(2)]
for every (inhomogeneous) linear polynomial $f\in\R[x]$ with $f\ge0$
on $S$, the element $\phi^*(f)$ of $\R[X]$ is a sum of squares of
elements from~$U$.
\end{itemize}
\end{thm}

\begin{proof}
If there exist $\phi$ and $U$ satisfying (1) and (2), $\ol K$ has a
semidefinite representation
by Proposition \ref{umker}. Conversely, assume that $\ol K$ has a
semidefinite representation,
and let $K_1=\{1\}\times K\subset\R\times\R^n=\R^{n+1}$. Since
$\ol K_1$ is a spectrahedral shadow, it is easy to see that
$\cone(\ol K_1)$ is a spectrahedral shadow
in $\R^{n+1}$.
Hence the closure of $\cone(\ol K_1)$ is a spectrahedral shadow
as well. Clearly, this last cone coincides with $\ol C_1$, where
$C_1$ is the convex cone in $\R^{n+1}$ generated by $S_1=\{1\}\times
S$.
Now we can apply the ``only if'' part of Theorem \ref{mainhom} to
$S_1$ and $C_1$ and deduce the converse in Theorem \ref{maininhom}.
\end{proof}

\begin{rem}\label{sharpenfurther}%
In \ref{maininhom} we may sharpen conditions (1) and (2) further.
Assume we are given a morphism $\phi\colon X\to\A^n$ and a linear
subspace $U\subset\R[X]$ as in \ref{maininhom}. From (1) we deduce
that there exists a \sa\ set $M\subset X(\R)$ with $\phi(M)=S$ and
with $\dim(M)=\dim(S)$ (use a semi-algebraic section $S\to X(\R)$
of $\phi$ over $S$). Let $X'$ be the Zariski closure of $M$ in $X$,
and let $U'\subset\R[X']$ be the image of $U$ under $\R[X]\to\R[X']$.
Then (1) and (2) hold as well for the restriction $\phi'\colon
X'\to\A^n$ of $\phi$ and for $U'$. Therefore, we can achieve in
addition that $\dim(X)=\dim(S)$.
On the other hand, condition (1) can be replaced by either
$\ol K=\phi(X(\R))$ (to make it seemingly stronger), or by
$K\subset\ol{\conv(\phi(X(\R)))}$ (to make it seemingly weaker). Note
that the inclusion $\phi(X(\R))\subset\ol K$ holds for any $\phi$
satisfying~(2). Indeed, given any $\xi\in\R^n$, $\xi\notin\ol K$,
there exists $f\in\R[x]$ linear with $f|_S\ge0$ and $f(\xi)<0$, so
(2) implies $\xi\notin\phi(X(\R))$.
\end{rem}

\begin{lab}\label{xtendualcone}%
In the next section we need to work not only over the field $\R$ of
real numbers, but also over real closed extension fields $R\supset
\R$. Given an affine $\R$-variety $V$ and a \sa\ set $M\subset
V(\R)$, the base field extension of $M$ to $R$ is denoted $M_R$
(see \cite[Section 5.1]{bcr}). If $M$ is described by a finite
boolean
combination of inequalities $f_i>0$ (with $f_i\in\R[V]$), the set
$M_R\subset V(R)$ is described by the same system of inequalities.
\end{lab}

\begin{rem}\label{uniform}%
Let $\phi\colon X\to V$ be a morphism of affine $\R$-varieties, let
$L\subset\R[V]$ and $U\subset\R[X]$ be finite-dimensional linear
subspaces, and let $S\subset V(\R)$ be a \sa\ set. Assume that the
following condition holds:
\begin{quote}
$(*)$
For every $f\in L$ with $f\ge0$ on $S$, the pull-back
$\phi^*(f)\in\R[X]$ is a sum of squares of elements of $U$.
\end{quote}
Then the extension of $(*)$ to any real closed field extension $R$ of
$\R$ holds as well. More precisely, any $f\in L_R=L\otimes R\subset
R[V]$ with $f\ge0$ on $S_R\subset V(R)$ becomes a sum of squares of
elements of $U_R=U\otimes R$ in $R[X]$, by the Tarski principle.

In particular, any $f\in L\otimes R$ with $f\ge0$ on $S_R$ becomes a
sum of squares in $R[X]$. We remark that this last conclusion would
fail in general if in $(*)$ we had only required that $\phi^*(f)$ is
a sum of squares in $\R[X]$. For instance, taking $\phi$ to be the
identity of $X=V=\A^2$ and $S$ the unit disk would give
counter-examples: Every $f\in\R[x_1,x_2]$ with $f\ge0$ on $S$ can be
written $f=p+(1-x_1^2-x_2^2)q$ with sums of squares $p,\,q\in
\R[x_1,x_2]$, but the analogous statement fails over any proper real
closed extension $R$ of~$\R$ (see \cite{sch:surf} and
\cite{sch:stable}).
Rather, one needs that uniform sums of squares expressions exist as
in $(*)$, to guarantee that the condition is stable under real closed
field extension.
\end{rem}

The following version is essentially identical with the ``only if''
part of Theorem \ref{maininhom}, but will be more convenient in the
next section.
Let $V$ be an affine $\R$-variety, and let $L\subset\R[V]$ be a
finite-dimensional linear subspace. Let $\varphi_L\colon V\to\A_L
\cong\A^n$ ($n=\dim(L)$) be the associated morphism, see
\ref{linsystmorph}.

\begin{cor}\label{handlicher}%
With $V$ and $L$ as above, let $S\subset V(\R)$ be a \sa\ set. Assume
that $\ol{\conv(\varphi_L(S))}$, the closed convex hull in $\A_L(\R)=
L^\du$, is a spectrahedral shadow.
Then there exists a
morphism $\phi\colon X\to V$ of affine $\R$-varieties, together with
a finite-dimensional linear subspace $U\subset\R[X]$, such that
$S\subset\phi(X(\R))$ and the following holds:
\begin{quote}
For every real closed field $R\supset\R$ and every $f\in L_R+R1
\subset R[V]$ with $f\ge0$ on $S_R$, the pull-back $\phi_R^*(f)$
under $\phi_R\colon X_R\to V_R$ is a sum of squares of elements from
$U\otimes R$ in $\R[X]\otimes R=R[X]$.
\end{quote}
(The converse is true as well, covered by Proposition \ref{umker}.)
\end{cor}

\begin{proof}
By \ref{maininhom} there exists a morphism $\psi\colon Y\to\A_L$ of
affine $\R$-varieties, together with a finite-dimensional subspace
$W\subset\R[Y]$, such that, for every $f\in\R1+L$ with $f\ge0$ on
$S$, the
pull-back $\psi^*(f)\in\R[Y]$ is a sum of squares of elements from
$W$. Let $X$ be the fibered product of $V$ and $Y$ over $\A_L$, let
$\phi\colon X\to V$ be the canonical morphism, and let $U\subset
\R[X]$ be the pull-back of $W$ under $X\to Y$. Then the condition in
\ref{handlicher} is satisfied for $R=\R$. By Tarski-Seidenberg, the
condition holds over any real closed extension $R$ as well (see also
Remark \ref{uniform}).
\end{proof}

It may not be obvious immediately, but the necessary condition for
semidefinite representability found in Theorems \ref{mainhom} resp.\
\ref{maininhom} is quite restrictive.
In the next section we'll elaborate on this in more
detail.


\section{Constructing examples}\label{counterex}%

We use properties of smooth morphisms of algebraic varieties,
together with a weak version of generic smoothness, to construct
examples of convex sets that have no semidefinite representation.

\begin{lab}
Let $k$ be a field. Recall that a morphism $\phi\colon X\to Y$ of
algebraic $k$-varieties is smooth at $x\in X$ if there exist affine
open sets $U=\Spec(A)\subset X$ and $V=\Spec(B)\subset Y$ with $x\in
U$ and $\phi(U)\subset V$ such that $A$ is (via $\phi$)
$B$-isomorphic to $B[x_1,\dots,x_n]/(f_1,\dots,f_m)$, where $m\le n$
and $\det(\partial f_i/\partial x_j)_{1\le i,j\le m}$ is a unit in
$\scrO_{X,x}$. It is equivalent that $\phi$ is flat at $x$ and that
the fibre $\phi^{-1}(\phi(x))$ is geometrically regular at $x$ over
the residue field of $\phi(x)$, see \cite[17.5.1]{ega4}.
The smooth locus of $\phi$, i.e.\ the set of points $x\in X$ at
which $\phi$ is smooth, is Zariski open in $X$.

We will use the following weak version of generic smoothness (compare
\cite[Lemma III.10.5]{ha}):
\end{lab}

\begin{prop}\label{gensmoth}%
Let $\phi\colon X\to Y$ be a dominant morphism between irreducible
$k$-varieties where $\ch(k)=0$. Then there exists a non-empty Zariski
open subset $U\subset X$ such that $\phi|_U\colon U\to Y$ is smooth.
\end{prop}

The following result is contained in \cite[17.5.3]{ega4} as a
particular case:

\begin{prop}\label{smothpowser}%
Let $\phi\colon X\to Y$ be a morphism of algebraic $k$-varieties.
Let $\xi\in X(k)$, and write $A=\scrO_{X,\xi}$,
$B=\scrO_{Y,\phi(\xi)}$. Then $\phi$ is smooth at $\xi$ if and
only if $\wh A$ is isomorphic over $\wh B$ to a power series algebra
$\wh B[[t_1,\dots,t_m]]$.
\end{prop}

(Here, of course, hat denotes completion of a local ring.) From
\ref{smothpowser} we deduce the following observation:

\begin{lem}\label{sosglatt}%
Let $\phi\colon X\to Y$ be a morphism of algebraic $k$-varieties, and
assume that $\phi$ is smooth at $\xi\in X(k)$. If $f\in
\scrO_{Y,\phi(\xi)}$ is such that $\phi^*(f)$ is a sum of squares in
$\wh\scrO_{X,\xi}$, then $f$ is a sum of squares in
$\wh\scrO_{Y,\phi(\xi)}$.
\end{lem}

\begin{proof}
Indeed, if an element of a ring $B$ becomes a sum of squares in
$B[[t_1,\dots,t_m]]$, it was already a sum of squares in $B$.
\end{proof}

\begin{lab}\label{standasspts}%
We shall present two constructions. Each will give us concrete
examples of convex \sa\ sets without semidefinite representation. For
both, the reasoning will be based on the following technical lemma.
We will repeatedly assume that data is given as follows:

$(*)$
$V$ is an affine $\R$-variety, $L\subset\R[V]$ is a
finite-dimensional linear subspace, $\varphi_L\colon V\to\A_L\cong
\A^m$ ($m=\dim(L)$) is the associated morphism (see
\ref{linsystmorph}), and $S\subset V(\R)$ is a \sa\ set. Moreover,
$V'$ is an irreducible component of $V$ and $S'\subset S\cap
V'(\R)$ is a \sa\ set, Zariski-dense in $V'$.

(Note that some of the technicalities in $(*)$ and in \ref{lemconst}
arise since we want to cover sets $S$ as well whose Zariski closure
has several irreducible components. Otherwise we could have assumed
$V'=V$ and $S'=S$.)
\end{lab}

\begin{lem}\label{lemconst}%
Assume that $(*)$ as in \ref{standasspts} is given.
If $\ol{\conv(\varphi_L(S))}$ is a spectrahedral shadow
in $\A_L(\R)$, there
exists a morphism $\psi\colon W\to V'$ of affine $\R$-varieties,
together with $\xi\in W(\R)$, such that the following hold:
\begin{itemize}
\item[(1)]
$W(\R)$ is Zariski dense in $W$,
\item[(2)]
$\psi(\xi)\in S'$,
\item[(3)]
$\psi$ is smooth at $\xi$,
\item[(4)]
for every real closed field $R\supset\R$ and every $f\in L_R+R1
\subset R[V]$ with $f\ge0$ on $S_R$, the pull-back $\psi_R^*(f)\in
R[W]$ is a sum of squares in $R[W]$.
\end{itemize}
\end{lem}

In (4) we have written $L_R=L\otimes R$, which is a
finite-dimensional $R$-linear subspace of $\R[V]\otimes R=R[V]$.

\begin{proof}
By Corollary \ref{handlicher}, there exists a morphism $\phi\colon
X\to V$ of affine $\R$-varieties with $S\subset\phi(X(\R))$ such
that, for every real closed $R\supset\R$ and every $f\in L_R+R1
\subset R[V]$ with $f\ge0$ on $S_R$, the pull-back $\phi_R^*(f)$ is a
sum of squares in $R[X]$.
Using the argument of Remark \ref{sharpenfurther}, we can find a
closed irreducible subvariety $X'$ of $X$ satisfying $\phi(X')\subset
V'$ and $\dim(X')=V'$, for which $S'\cap\phi(X'(\R))$ is Zariski
dense in $V'$.
The restriction $\phi'\colon X'\to V'$ of $\phi$ is a dominant
morphism between irreducible $\R$-varieties of the same dimension.
By Proposition
\ref{gensmoth}, there is a non-empty open affine subset $W$ of $X'$
such that the restriction $\phi'|_W\colon W\to V'$ of $\phi'$
is smooth. Writing $Z=X'\setminus W$ we have $\dim(Z)<\dim(V')$, so
the set $\phi'(Z(\R))$ is not Zariski dense in $V'$. Therefore
$S'\cap\phi'(W(\R))$ is still Zariski dense in~$V'$.
In particular, we can find $\xi\in W(\R)$ such that $\eta:=
\phi'(\xi)$ lies in $S'$. Then it is clear that (1)--(4) are
satisfied for $\psi:=\phi'|_W\colon W\to V'$ and~$\xi$.
\end{proof}

The first construction is very easy and works for convex hulls of
suitable sets of dimension~$\ge3$. First recall:

\begin{lem}\label{tamslem}%
Let $A$ be a regular local $\R$-algebra, let $p_1,\dots,p_d$ be a
regular system of parameters of $A$. If $f(x_1,\dots,x_d)$ is a
form in $d$ variables over $\R$ that is not a sum of squares of
forms, then $f(p_1,\dots,p_d)\in A$ is not a sum of squares in~$A$.
\end{lem}

The proof uses the associated graded ring of $A$, see
\cite{sch:tams}, proof of Proposition 6.1.

\begin{prop}\label{1stconstneu}%
Assume that $(*)$ as in \ref{standasspts} is given. If for every
$\eta\in S'$ there exists $f\in L+\R1\subset\R[V]$ with $f|_S\ge0$
such that $f$ is not a sum of squares in $\wh\scrO_{V,\eta}$, then
the closed convex hull $\ol{\conv(\varphi_L(S))}$
in $\A_L(\R)\cong\R^{\dim(L)}$ fails to be a spectrahedral shadow.
\end{prop}

\begin{proof}
Assume that the closed convex hull is a spectrahedral shadow.
Then there exists a
morphism $\psi\colon W\to V'$ together with a point $\xi\in W(\R)$ as
in Lemma \ref{lemconst}. Let $\eta=\psi(\xi)\in S'$, and choose $f\in
L+\R1$ for the given $\eta$ as in the hypothesis. On the one hand,
$\psi^*(f)\in\R[W]$ should be a sum of squares in $\R[W]$, by
property (4) of $\psi$ in \ref{lemconst}. On the other hand, since
$\psi$ is smooth at $\xi$, this contradicts Lemma \ref{sosglatt}, by
the choice of $f$.
\end{proof}

\begin{example}\label{3dex}%
Let $x=(x_1,x_2,x_3)$ and put $L=\{f\in\R[x]\colon\deg(f)\le6$,
$f(0)=0\}$, a linear subspace of $\R[x]$ with $\dim(L)=83$.
For every $\xi\in\R^3$ there exists $f\in L+\R1$ with $f\ge0$ on
$\R^3$ such that $f$ is not a sum of squares in
$\wh\scrO_{\A^3,\xi}$ (the ring of formal power series in
$x_1-\xi_1$, $x_2-\xi_2$, $x_3-\xi_3$). Indeed, this follows from
\ref{tamslem}, e.g.\ by taking $f=p(x_1-\xi_1,x_2-\xi_2,x_3-\xi_3)$
where $p$ is any ternary sextic form that is psd but not a sum of
squares (for instance the Motzkin form). Let
$\varphi=\varphi_L\colon\A^3\to\A_L\cong
\A^{83}$ be the Veronese type embedding associated with $L$. For any
\sa\ set $S\subset\R^3$ with non-empty interior, it follows from
Proposition \ref{1stconstneu} that the closed convex hull of
$\varphi(S)$ in $\R^{83}$ has no semidefinite representation.
\end{example}

\begin{example}\label{4dex}%
Similarly, let $x=(x_1,x_2,x_3,x_4)$ and $L=\{f\in\R[x]\colon
\deg(f)\le4$, $f(0)=0\}$. Then $\dim(L)=69$.
Using psd, non-sos quartic forms in four variables and proceeding
similarly as in \ref{3dex}, we find that the closed convex hull of
$\varphi_L(S)$ in $\R^{69}$ is not a spectrahedral shadow,
for any \sa\ set $S\subset\R^4$ with nonempty interior.
\end{example}

\begin{rem}\label{parsim1}%
The reasoning used in the preceding examples was still very coarse.
With
a finer look we arrive at constructions that are considerably more
parsimonious. For example, if in \ref{3dex} we work with the Motzkin
form $p$, we can find a linear subspace $L\subset\R[x]$ of dimension
$\dim(L)=27$ such that $p(x-\xi)\in\R1+L$ for every $\xi\in\R^3$,
resulting in an embedding $\R^3\to\R^{27}$ with the property of
\ref{3dex}. Similarly, if in \ref{4dex} we work with the Choi-Lam
form $p(x)=x_1^2x_2^2+x_2^2x_3^2+x_3^2x_1^2+x_4^4-4x_1x_2x_3x_4$,
we can find a linear subspace of dimension~$19$ with the desired
property.
\end{rem}

\begin{rem}
The construction of convex sets without semidefinite representation
via Proposition
\ref{1stconstneu} did not employ the full strength of the ``only~if''
part of Theorem \ref{maininhom}. Indeed, it wasn't used anywhere that
pull-backs of non-negative linear polynomials are \emph{uniformly}
sums of squares in $\R[X]$ (see Remark \ref{uniform}). In turn, the
argumentation in \ref{3dex}\,--\,\ref{parsim1} applies only to convex
hulls of sets of dimension at least three. We now refine the
construction. This will provide us with convex hulls of
two-dimensional sets without a semidefinite representation.
\end{rem}

\begin{lab}\label{rbetc}%
Let $R$ be a real closed field containing $\R$. In the sequel, we
always denote by $B$ the convex hull of $\R$ in $R$, so $B=\{a\in
R\colon\ex n\in\N$ $-n<a<n\}$. Note that $B$ is a valuation ring
(called the canonical valuation ring of $R$), with field of fractions
$R$, maximal ideal $\m_B=\{a\in R\colon\all n\in\N$ $|na|<1\}$ and
residue field $B/\m_B=\R$. The reduction map $B\to B/\m_B=\R$ will be
denoted $a\mapsto\ol a$. Nonzero elements in $\m_B$ will be called
\emph{infinitesimals} of~$R$.

An example is given by the field $R=\bigcup_{q\ge1}\R((t^{1/q}))$ of
Puiseux series with real coefficients (see \cite[Section 2.6]{bpr}).
The sign of $0\ne f=\sum_{m\ge m_0}c_mt^{m/q}\in R$ with $c_m\in\R$
and $c_{m_0}\ne0$ is the sign of $c_{m_0}$, and the
valuation (or order) of $f$ is $o(f)=\frac{m_0}q$. The valuation ring
$B$ resp.\ its maximal ideal $\m_B$ consists of all $f\in R$ with
$o(f)\ge0$ resp.\ $o(f)>0$.

Let $V$ be an affine $\R$-variety. We write $B[V]:=\R[V]\otimes B$
(tensor product over~$\R$).
If $\phi\colon X\to V$ is a morphism of affine $\R$-varieties, then
$\phi_R^*$ (resp.\ $\phi_B^*$) denotes the induced homomorphism
$R[V]\to R[X]$ (resp.\ $B[V]\to B[X]$). Given $\xi\in V(\R)$, let
$M_{V,\xi}\subset B[V]$ be the kernel of the evaluation map $B[V]\to
B$, $f\mapsto f(\xi)$.

We start with several auxiliary results. The following lemma is
straightforward:
\end{lab}

\begin{lem}\label{straightf}%
Let $V$ be an affine $\R$-variety, let $\xi\in V(\R)$, and let
$R,\,B$ as in \ref{rbetc}. Then for every $N\ge1$ the natural
map
$$\bigl(\scrO_{V,\xi}/(\m_{V,\xi})^N\bigr)\otimes B\>\to\>
B[V]/(M_{V,\xi})^N$$
of $B$-algebras is an isomorphism.
\qed
\end{lem}

\begin{lem}\label{relred}%
Let $R,\,B$ be as in \ref{rbetc}, and let $X$ be an affine
$\R$-variety for which $X(\R)$ is Zariski dense in $X$. If
$g_1,\dots,g_r\in R[X]$ are such that $\sum_{i=1}^rg_i^2$ lies in
$B[X]$, then $g_i\in B[X]$ for every~$i$.
\end{lem}

\begin{proof}
We can assume $g_i\ne0$ for every~$i$. Let $f:=\sum_{i=1}^rg_i^2$.
There is $0\ne c\in R$ such that $cg_i\in B[X]$ for every $i$ and
$\ol{cg_j}\ne0$ in $(B/\m_B)[X]=\R[X]$ for at least one index~$j$.
It follows that $c^2f\in B[X]$, and moreover $\ol{c^2f}=\sum_i
(\ol{cg_i})^2$ is nonzero in $(B/\m_B)[X]=\R[X]$, since $X(\R)$ is
Zariski dense in $X$.
Hence $c\notin\m_B$,
which means that $\frac1c\in B$, and so indeed $g_i\in
B[X]$ for every index~$i$.
\end{proof}

\begin{lem}\label{antwort6}%
Let $R,\,B$ be as in \ref{rbetc}, and let $\phi\colon X\to V$ be a
morphism of affine $\R$-varieties.
Assume that $X(\R)$ is Zariski dense in $X$, and that $\phi$ is
smooth at $\xi\in X(\R)$. If $f\in B[V]$ and $N\ge1$ are such that
$f$ is not a sum of squares in $B[V]$ modulo $(M_{V,\phi(\xi)})^N$,
then $\phi_R^*(f)\in R[X]$ is not a sum of squares in~$R[X]$.
\end{lem}

\begin{proof}
Write $\eta=\phi(\xi)$. By Proposition \ref{smothpowser}, the
smoothness assumption implies that the completed local ring
$\wh\scrO_{X,\xi}$ is $\wh\scrO_{V,\eta}$-isomorphic to a power
series ring over $\wh\scrO_{V,\eta}$. In particular, this implies
that $\phi^*\colon\scrO_{V,\eta}/(\m_{V,\eta})^N\to\scrO_{X,\xi}/
(\m_{X,\xi})^N$ has a retraction, i.e.\ there is a homomorphism
$\rho\colon\scrO_{X,\xi}/(\m_{X,\xi})^N\to\scrO_{V,\eta}/
(\m_{V,\eta})^N$ for which the composition $\rho\comp\phi^*$ is the
identity on $\scrO_{V,\eta}/(\m_{V,\eta})^N$.
Tensoring with $B$ and using Lemma \ref{straightf} gives the
commutative diagram
$$\begin{xy}
\xymatrix{%
B[V] \ar[r]^{\phi_B^*} \ar[d] & B[X] \ar[d] \\
B[V]/(M_{V,\eta})^N \ar[r] & B[X]/(M_{X,\xi})^N}\end{xy}$$
whose bottom map has a retraction. From the hypothesis it therefore
follows that $\phi^*_B(f)\in B[X]$ cannot be a sum of squares in
$B[X]$. By Lemma \ref{relred}, $\phi_R^*(f)$ is not a sum of squares
in $R[X]$ either.
\end{proof}

\begin{lem}\label{trunclocringotimes}%
Let $R,\,B$ be as in \ref{rbetc},
let $V$ be an affine $\R$-variety, and let $\xi\in V(\R)$ be a
nonsingular $\R$-point. If $u_1,\dots,u_d\in\R[V]$ form a regular
parameter sequence of $V$ at $\xi$, there is an isomorphism
$$B[V]/(M_{V,\xi})^N\>\cong\>B[x_1,\dots,x_d]/\langle x_1,\dots,x_d
\rangle^N$$
of $B$-algebras which makes the cosets of $u_i$ and $x_i$ correspond
to each other, for $i=1,\dots,d$.
\end{lem}

\begin{proof}
Clear from the isomorphism $\R[[x_1,\dots,x_d]]\to\wh\scrO_{V,\xi}$
sending $x_i$ to $u_i$, and from Lemma \ref{straightf}.
\end{proof}

The next result is a key observation. For $R\ne\R$ a proper real
closed field extension of $\R$, it implies that there exist
polynomials $f\in B[x_1,x_2]$ with $f\ge0$ on $R^2$ such that $f$ is
not a sum of squares in $B[x_1,x_2]/\langle x_1,x_2\rangle^N$, for
$N$ sufficiently large. Note that any such $f$ is a sum of squares in
$R[[x_1,x_2]]$, and hence in $R[x_1,x_2]/\langle x_1,x_2\rangle^N$
for all $N$~\cite{sch:local}.

\begin{prop}\label{antwort}%
Let $f\in\R[x_0,x]=\R[x_0,\dots,x_n]$ be homogeneous of degree~$d$,
and assume that $f$ is not a sum of squares in $\R[x_0,x]$. Let
$R,\,B$ be as in \ref{rbetc}. If $\epsilon>0$ is an infinitesimal
in $R$, the polynomial $f(\epsilon,x)\in B[x]$ is not a sum of
squares in $B[x]\big/\langle x_1,\dots,x_n\rangle^{d+1}B[x]$.
\end{prop}

\begin{proof}
Assume we have an identity $f(\epsilon,x)+g(x)=\sum_jp_j(x)^2$ where
$g(x)\in\langle x\rangle^{d+1}B[x]$ and $p_j(x)\in B[x]$. Replacing
$x$ by $\epsilon x$ yields
$$\epsilon^df(1,x)+g(\epsilon x)\>=\>\sum_jp_j(\epsilon x)^2.
\eqno(*)$$
The left hand side is divisible by $\epsilon^d$ in $B[x]$. By Lemma
\ref{relred}, the polynomial $q_j(x):=\epsilon^{-d/2}p_j(\epsilon x)
\in R[x]$ lies in $B[x]$ for every~$j$.
Putting $g'(x)=\epsilon^{-(d+1)}g(\epsilon x)$ we have $g'(x)\in
B[x]$, therefore dividing $(*)$ by $\epsilon^d$ gives
$$f(1,x)+\epsilon g'(x)\>=\>\sum_jq_j(x)^2,$$
an identity in $B[x]$. Reducing coefficient-wise modulo $\m_B$
implies that $f(1,x)$ is a sum of squares in $\R[x]$, contradicting
the hypothesis.
\end{proof}

\begin{prop}\label{2ndconstneu}%
Assume that $(*)$ as in \ref{standasspts} is given, and assume that
$R\supset\R$, $R\ne\R$ is a real closed field with canonical
valuation ring $B$ (\ref{rbetc}). For every $\eta\in S'$ assume that
there exists $f\in L_B+B1\subset B[V]$ with $f\ge0$ on $S_R$ such
that $f$ is not a sum of squares in $B[V]/(M_{V,\eta})^N$ for some
$N\ge1$. Then the closed convex hull $\ol{\conv(\varphi_L(S))}$ in
$\A_L(\R)\cong\R^{\dim(L)}$ is not a spectrahedral shadow.
\end{prop}

(Here $L_B:=L\otimes B\subset\R[V]\otimes B=B[V]$.)

\begin{proof}
Assume that the closed convex hull is a spectrahedral shadow.
Then there exists
$\psi\colon W\to V'$ together with $\xi\in W(\R)$, as in Lemma
\ref{lemconst}. Let $\eta=\psi(\xi)\in S'$, and choose $f\in L_B+B1$
for the given $\eta$ as in \ref{2ndconstneu}. On the one hand,
$\psi_R^*(f)\in R[W]$ should be a sum of squares in $R[W]$, by
property (4) of $\psi$ in \ref{lemconst}. On the other hand,
$\psi_R^*(f)$ is not a sum of squares in $R[W]$ by Lemma
\ref{antwort6}. This contradiction proves Proposition
\ref{2ndconstneu}.
\end{proof}

\begin{example}\label{ex27}%
Let $x=(x_1,x_2)$, and put $L=\{f\in\R[x]\colon\deg(f)\le6$,
$f(0)=0\}$, a linear subspace of $\R[x]$ of dimension~$27$.
Consider the associated embedding $\phi_L\colon\A^2\to\A_L\cong
\A^{27}$. If $S\subset\R^2$ is any \sa\ set with non-empty interior,
the closed convex hull of $\varphi_L(S)$ in $\R^{27}$ is not a
spectrahedral shadow.
Indeed, choose a sextic form $p\in\R[x_0,x_1,x_2]$ that is psd
but not a sum of squares, and let $0\ne\epsilon$ be an infinitesimal
of $R$. Given $\xi\in S$, the polynomial $f:=p(\epsilon,x_1-\xi_1,
x_2-\xi_2)\in B[x_1,x_2]$ lies in $L_B+B1$, and $f$ is not a sum of
squares in $B[x]/(M_{\A^2,\xi})^7$
by Proposition \ref{antwort} and Lemma \ref{trunclocringotimes}.
It follows from Proposition \ref{2ndconstneu} that
$\ol{\conv(\varphi_L(S))}$ has no semidefinite representation.
\end{example}

\begin{rem}\label{parsim2}%
Similar to Remark \ref{parsim1}, we can arrive at examples of smaller
dimension when we take a finer look. For instance, consider the
Motzkin form $p=x_1^6+x_0^4x_2^2+x_0^2x_2^4-3x_0^2x_1^2x_2^2$ in
$\R[x_0,x_1,x_2]$. This form is psd but not a sum of squares. Let
$L\subset\R[x,y]$ be the linear space spanned by the 14 monomials
$x^i$ ($1\le i\le6$), $y^i$ ($1\le i\le4$) and $x^iy^j$ ($i,j=1,2$).
For any semi-algebraic set $S\subset\R^2$ with non-empty interior,
the closed convex hull of $\varphi_L(S)$ in $\R^{14}$ fails to be a
spectrahedral shadow.
Indeed, for any choice of $\epsilon,\,\xi_1,\,
\xi_2\in B$, the polynomial $f:=p(\epsilon,\,x_1-\xi_1,\,x_2-\xi_2)
\in B[x_1,x_2]$ lies in $L_B+B1$. So we can argue as in \ref{ex27}.
\end{rem}

A reasoning as in Examples \ref{ex27} or \ref{parsim2} can also be
applied to affine $\R$-varieties $V$ different from $\A^n$, thanks to
the following lemma:

\begin{lem}\label{2ndconstnotsos}%
Let $R,\,B$ be as in \ref{rbetc}, and assume $R\ne\R$.
Let $V$ be an affine $\R$-variety, let $V_\reg\subset V$ be its
smooth locus, let $\eta\in V_\reg(R)$, and
let $q_1,\dots,q_n\in B[V]$ be a regular parameter sequence for
$\scrO_{V_R,\eta}$. Moreover let $f\in\R[x_0,\dots,x_n]$ be a form
that is psd but not a sum of squares. If $\epsilon\ne0$ is any
infinitesimal in $R$, then $f(\epsilon,q_1,\dots,q_n)\in B[V]$ is psd
on $V(R)$, but is not a sum of squares in $B[V]/(M_{V,\eta})^N$ for
$N\ge\deg(f)+1$.
\end{lem}

\begin{proof}
Put $p:=f(\epsilon,q_1,\dots,q_n)$. It is clear that $p\ge0$ on
$V(R)$.
Let $x=(x_1,\dots,x_n)$. We have an isomorphism
$B[x]/\langle x\rangle^N\isoto B[V]/(M_{V,\eta})^N$ for every
$N\ge1$, that sends $x_i$ to $q_i$ for $i=1,\dots,n$ (Lemma
\ref{trunclocringotimes}). It maps the residue class of
$f(\epsilon,x)$ to the residue class of $p$. By Proposition
\ref{antwort}, this element (in either ring) is not a sum of
squares when $N>\deg(f)$.
\end{proof}

Summing up, we can conclude:

\begin{thm}
Let $S\subset\R^m$ be any \sa\ set with $\dim(S)\ge2$. Then, for
some $k\ge1$, there exists a polynomial map $\varphi\colon S\to
\R^k$ such that the closed convex hull of $\varphi(S)$ in $\R^k$ has
no semidefinite representation.
\end{thm}

\begin{proof}
Let $V\subset\A^m$ be the Zariski closure of $S$. Fix a point $\xi\in
S\cap V_\reg(\R)$ such that $\dim_\xi(S)\ge2$ and $S$ contains an
open neighborhood of $\xi$ in $V(\R)$.
Let $p_1,\dots,p_n\in\R[V]$ ($n\ge2$) be a regular sequence of
parameters for $\scrO_{V,\xi}$. Let $x=(x_1,\dots,x_n)$,
$y=(y_1,\dots,y_n)$ be tuples of variables, let $f\in\R[t,x]$ be a
form in $n+1$ variables that is psd but not a sum of squares, and put
$d=\deg(f)$. We can write
$$f(t,x+y)\>=\>\sum_{i=0}^dg_i(x)\,h_{d-i}(t,y)$$
where $g_i\in\R[x]$ and $h_i\in\R[t,y]$ are forms of degree~$i$
($i=0,\dots,d$). There is a Zariski open neighborhood $U\subset
V_\reg$ of $\xi$ such that, for any $\eta\in U(\R)$, the sequence
$p_i-p_i(\eta)$ ($i=1,\dots,n$) is a regular sequence of parameters
for $\scrO_{V,\eta}$.
Let $L\subset\R[V]$ be a finite-dimensional linear subspace that
contains $g_i(p_1,\dots,p_n)$ for $i=1,\dots,d$, and choose a real
closed field $R$ that properly contains $\R$. For any
$a=(a_0,\dots,a_n)\in B^{n+1}$, the element
$$q_a\>:=\>f(a_0,p_1+a_1,\dots,p_n+a_n)\>=\>\sum_{i=0}^d
g_i(p_1,\dots,p_n)\,h_{d-i}(a_0,\dots,a_n)$$
lies in $L_B+B1\subset B[V]$ and satisfies $q_a\ge0$ on $V(R)$.
Let $\eta\in U(\R)$, and put $a=(\epsilon,\,-p_1(\eta),\,\dots,\,
-p_n(\eta))\in B^{n+1}$ where $\epsilon\ne0$ is infinitesimal in~$R$.
Then $q_a\in B[V]$ is non-negative on $V(R)$, and $q_a$ is not a sum
of squares in $B[V]/(M_{V,\eta})^{d+1}$ by Lemma
\ref{2ndconstnotsos}.
By Proposition \ref{2ndconstneu}, this shows that
$\ol{\conv(\varphi(S))}$ is not a spectrahedral shadow.
\end{proof}

The previous examples already indicate that convex hulls of Veronese
sets typically fail to be spectrahedral shadows.
Specifically, we have:

\begin{cor}\label{converonese}%
Let $n,\,d$ be positive integers with $n\ge3$ and $d\ge4$, or with
$n=2$ and $d\ge6$. Let $m_1,\dots,m_N$ be the non-constant monomials
of degree $\le d$ in $(x_1,\dots,x_n)$ (so $N=\choose{n+d}n-1$).
Then for any \sa\ set $S\subset\R^n$ with non-empty interior, the
closed convex hull of
$$v(S)\>:=\>\Bigl\{\bigl(m_1(\xi),\dots,m_N(\xi)\bigr)\colon
\xi\in S\Bigr\}$$
in $\R^N$ fails to be a spectrahedral shadow.
\end{cor}

\begin{proof}
Hilbert \cite{hi} showed that there exists a psd form $f$ of degree
$d$ in $n+1$ variables. So it suffices to apply Propositions
\ref{antwort} and \ref{2ndconstneu}.
\end{proof}

For positive integers $n,\,d$ let $\Sigma_{n,2d}$ (resp.\
$P_{n,2d}$) denote the cone of all degree~$2d$ forms in
$\R[x_1,\dots,x_n]$ that are sums of squares of forms (resp.\ that
are positive semidefinite).

\begin{cor}\label{psdcone}%
The psd cone $P_{n,2d}$ is a spectrahedral shadow
only in the
cases where $P_{n,2d}=\Sigma_{n,2d}$, i.e.\ only for $2d=2$ or
$n\le2$ or $(n,2d)=(3,4)$.
\end{cor}

\begin{proof}
It is well-known and easy to see that the dual $\Sigma_{n,2d}^*$ of
the sos cone is a spectrahedral cone. Therefore $\Sigma_{n,2d}$,
being closed, is a spectrahedral shadow.
Let $n,\,d$ be such
that $\Sigma_{n,2d}\ne P_{n,2d}$. By Hilbert's theorem \cite{hi}
quoted before, this happens precisely when $2d=2$ or $n\le2$ or
$(n,2d)=(3,4)$. The dual cone $P_{n,2d}^*$ can be identified with the
convex (or conical) hull of the image of the degree~$2d$ Veronese map
$$v_{n,2d}\colon\>\R^n\to\R^N,\quad\xi\mapsto
\bigl(\xi^\alpha\bigr)_{|\alpha|=2d}$$
where $N={\choose{n+2d-1}{n-1}}$ is the number of monomials of degree
$2d$ in $(x_1,\dots,x_n)$.
By \ref{converonese}, a suitable affine hyperplane section of this
cone fails to be a spectrahedral shadow.
So $P_{n,2d}^*$ itself cannot be a spectrahedral shadow,
and therefore neither can be $P_{n,2d}$.
\end{proof}


\section{Some open questions}

There are many obvious questions that remain open at this point. Here
are some that we consider as being particularly natural.

\begin{lab}\label{oq1}%
What is the smallest dimension of a convex \sa\ set without
semidefinite representation? The smallest dimension that we realize
in this paper is $14$ (see \ref{parsim2}). A more technical
construction gives examples of dimension~$11$. We expect that the
true answer should be much less. Is it three? Recall that the
Helton-Nie conjecture has been proved for subsets of $\R^2$
\cite{sch:sdpcurv}.
\end{lab}

\begin{lab}\label{oq2}%
Consider the necessary and sufficient condition \ref{mainhom} (or
\ref{maininhom}) for semidefinite
representability. Although we use it to
construct counter-examples to the Helton-Nie conjecture, it seems
that in concrete cases, the condition is often hard to decide. For a
prominent example let $C_n\subset\Sym_n(\R)$ be the \emph{copositive
cone}, consisting of all symmetric matrices $A$ such that
$x^tAx\ge0$ for all $x\in(\R_\plus)^n$ (see \cite{hiuse} for a
recent survey). For $n\ge5$ it is not known whether $C_n$ is a
spectrahedral shadow
(\cite[p.~135]{bpt}).
We were unable to apply criterion \ref{mainhom} to decide this
question.

Therefore we ask: What are alternative characterizations of
spectrahedral shadows
that are easier to work with?
\end{lab}

\begin{lab}\label{oq3}%
The results of Helton and Nie \cite{hn1}, \cite{hn2} guarantee the
existence of a semidefinite representation in a wide range of cases.
Specifically, if a compact convex \sa\ set $K\subset\R^n$ fails to
have a semidefinite representation, their results imply that the
boundary of $K$ must have a singular point, or must have zero
curvature somewhere (\cite{hn1}, conclusions, p.~790).

The counter-examples to the Helton-Nie conjecture constructed in this
paper are typically (closed) convex hulls of low-dimensional sets in
high-dimensional euclidean space. In particular, their boundaries
have singularities. It seems to be an open question whether there
exist counter-examples with smooth boundary.
\end{lab}

\begin{lab}\label{oq4}%
The generalized Lax conjecture (see \cite{vi} for an overview)
asserts that the hyperbolicity cone in $\R^n$ of any hyperbolic form
$f(x_1,\dots,x_n)$ is a spectrahedral cone. For $n=3$ this is in fact
a theorem, proved by Helton-Vinnikov \cite{hv} in 2007 in a
significantly stronger form. For $n\ge4$ however, it is not even
known in general whether every hyperbolicity cone is a spectrahedral
shadow,
although this holds when the cone is smooth (Netzer-Sanyal
\cite{ns}). Can one decide whether hyperbolicity cones are
spectrahedral shadows
using results of this paper?
\end{lab}



\begin{thebibliography}{28}

\bibitem{al}
M.~Anjos, J.\,B.~Lasserre (eds):
\emph{Handbook on Semidefinite, Conic and Polynomial Optimization}.
Springer, New York, 2012.

\bibitem{bpr}
S.~Basu, R.~Pollack, M.-F.~Roy:
\emph{Algorithms in Real Algebraic Geometry}.
2nd ed, Algorithms and Computation in Mathematics \textbf{10},
Springer, Berlin, 2006.

\bibitem{btn}
A.~Ben-Tal, A.~Nemirovski:
\emph{Lectures on Modern Convex Optimization}.
SIAM, Philadelphia, 2001.

\bibitem{bpt}
G.~Blekherman, P.~Parrilo, R.~Thomas (eds):
\emph{Semidefinite Optimization and Convex Algebraic Geometry}.
MOS-SIAM Series on Optimization \textbf{13}, SIAM, Philadelphia PA,
2013.

\bibitem{bcr}
J. Bochnak, M. Coste, M.-F.~Roy:
\emph{Real Algebraic Geometry}.
Erg.\ Math. Grenzgeb.\ (3) \textbf{36}, Springer, Berlin, 1998.

\bibitem{gn}
J. Gouveia, T. Netzer:
Positive polynomials and projections of spectrahedra.
SIAM J. Optim.\ \textbf{21}, 960--976 (2012).

\bibitem{gpt}
J. Gouveia, P. Parrilo, R. Thomas:
Theta bodies for polynomial ideals.
SIAM J. Optim.\ \textbf{20}, 2097--2118 (2010).

\bibitem{ega4}
A.~Grothendieck:
\emph{Elements de G\'eom\'etrie Alg\'ebrique}, Tome~4
(Quatri\`eme Partie).
Publ.\ Math.\ IHES \textbf{32}, 5--361 (1967).

\bibitem{gwz}
F.~Guo, C.~Wang, L.~Zhi:
Semidefinite representations of non-compact convex sets.
SIAM J.~Optim.\ \textbf{25}, 377--395 (2015).

\bibitem{ha}
R.~Hartshorne:
\emph{Algebraic Geometry}.
Grad.\ Texts Math.\ \textbf{52}, Springer, New York, 1977.

\bibitem{hn1}
W. Helton, J. Nie:
Sufficient and necessary conditions for semidefinite representability
of convex hulls and sets.
SIAM J. Optim.\ \textbf{20}, 759--791 (2009).

\bibitem{hn2}
W. Helton, J. Nie:
Semidefinite representation of convex sets.
Math.\ Program.\ \textbf{122} (Ser.~A), 21--64 (2010).

\bibitem{hv}
W.~Helton, V.~Vinnikov:
Linear matrix inequality representations of sets.
Comm.\ Pure Appl.\ Math.\ \textbf{60}, 654--674 (2007).

\bibitem{he}
D. Henrion:
Semidefinite representation of convex hulls of rational varieties.
Acta Appl.\ Math.\ \textbf{115}, 319--327 (2011).

\bibitem{hi}
D. Hilbert:
\"Uber die Darstellung definiter Formen als Summe von
Formenquadraten.
Math.\ Ann.\ \textbf{32}, 342--350 (1888).

\bibitem{hiuse}
J.-B.~Hiriart-Urruty, A.~Seeger:
A variational approach to copositive matrices.
SIAM Review \textbf{52}, 593--629 (2010).

\bibitem{la}
J.\,B. Lasserre:
Convex sets with semidefinite representation.
Math.\ Program.\ \textbf{120} (Ser.~A), 457--477 (2009).

\bibitem{la:book}
J.\,B. Lasserre:
\emph{Moments, Positive Polynomials and Their Applications}.
Imperial College Press, London, 2010.

\bibitem{ne}
A.~Nemirovski:
Advances in convex optimization: Conic programming.
Int.\ Cong.\ Math.\ vol.~I, European Math.\ Soc., Z\"urich, 2007,
pp.\ 413-–444.

\bibitem{nene}
Yu.~Nesterov, A.~Nemirovskii:
\emph{Interior-Point Polynomial Algorithms in Convex Programming}.
SIAM Studies in Applied Mathematics \textbf{13}, Philadelphia, 1994.

\bibitem{nt}
T. Netzer:
On semidefinite representations of non-closed sets.
Linear Algebra Appl.\ \textbf{432}, 3072--3078 (2010).

\bibitem{nps}
T. Netzer, D. Plaumann, M. Schweighofer:
Exposed faces of semidefinitely representable sets.
SIAM J.~Optim.\ \textbf{20}, 1944--1955 (2010).

\bibitem{ns}
T. Netzer, R. Sanyal:
Smooth hyperbolicity cones are spectrahedral shadows.
Math.\ Program.\ \textbf{153} (Ser.~B), 213--221 (2015).

\bibitem{nsi}
T. Netzer, R. Sinn:
A note on the convex hull of finitely many projections of
spectrahedra.
Preprint, \texttt{arxiv:0908.3386}.

\bibitem{ni}
J. Nie:
First order conditions for semidefinite representations of convex
sets defined by rational or singular polynomials.
Math.\ Program.\ \textbf{131} (Ser.~A), 1--36 (2012).

\bibitem{nipast}
J. Nie, P.~Parrilo, B.~Sturmfels:
Semidefinite representation of the $k$-ellipse.
In: \emph{Algorithms in Algebraic Geometry}, A.~Dickenstein,
F-O.~Schreyer, A.~Sommese (eds.), Springer, New York, 2008,
pp.\ 117--132.

\bibitem{pa}
P.~Parrilo:
Structured semidefinite programs and semialgebraic geometry methods
in robustness and optimization.
Ph.\,D.~Thesis, CalTech, 2000.

\bibitem{ra}
M.\,V. Ramana:
An exact duality theory for semidefinite programming and its
complexity implications.
Math.\ Prog.\ \textbf{77}, 129--162 (1997).

\bibitem{sp}
J. Saunderson, P.\,A. Parrilo:
Polynomial-sized semidefinite representations of derivative
relaxations of spectrahedral cones.
Math.\ Program., Ser.~A \textbf{153}, 309--331 (2015).

\bibitem{spw}
J. Saunderson, P.\,A. Parrilo, A.\,S. Willsky:
Semidefinite descriptions of the convex hull of rotation matrices.
SIAM J.~Optim.\ \textbf{25}, 1314--1343 (2015).

\bibitem{sch:tams}
C. Scheiderer:
Sums of squares of regular functions on real algebraic varieties.
Trans.\ Am.\ Math.\ Soc.\ \textbf{352}, 1039--1069 (1999).

\bibitem{sch:local}
C. Scheiderer:
On sums of squares in local rings.
J.~reine angew.\ Math.\ \textbf{540},205--227 (2001).

\bibitem{sch:surf}
C. Scheiderer:
Sums of squares on real algebraic surfaces.
Manuscr.\ math.\ \textbf{119}, 395--410 (2006).

\bibitem{sch:stable}
C. Scheiderer:
Non-existence of degree bounds for weighted sums of squares
representations.
J.~Complexity \textbf{21}, 823--844 (2005).

\bibitem{sch:conv}
C.~Scheiderer:
Convex hulls of curves of genus one.
Adv.\ Math.\ \textbf{228}, 2606--2622 (2011).

\bibitem{sch:sdpcurv}
C. Scheiderer:
Semidefinite representation for convex hulls of real algebraic
curves.
Preprint, 2012, \texttt{arxiv:1208.3865}, SIAM J. Applied Algebra
and Geometry (to appear).

\bibitem{vi}
V. Vinnikov:
LMI representations of convex semialgebraic sets and determinantal
representations of algebraic hypersurfaces: Past, present, and
future.
In: \emph{Mathematical Methods in Systems, Optimization, and
Control}, H.~Dym, M.\,C.~de Oliveira, M.~Putinar (eds.),
Springer Basel, pp. 325-–349 (2012).

\bibitem{wsv}
H.~Wolkowicz, R.~Saigal, L.~Vandenberghe (eds.):
\emph{Handbook of Semidefinite Programming. Theory, Algorithms, and
Applications}.
Kluwer, Boston, 2000.

\end{thebibliography}
\end{document}